\documentclass[11pt]{article}
\usepackage[utf8]{inputenc}
\usepackage{color}
\usepackage{subfigure}
\usepackage{url}
\usepackage{graphicx}
\usepackage{amsmath}
\usepackage{amsthm}
\usepackage{bm}
\usepackage[plainpages=false]{hyperref}

\usepackage{titlesec}
\titleformat{\section}
	{\normalfont\Large\bfseries\filcenter}{\thesection.}{1 ex}{}
\titleformat{\subsection}
	{\normalfont\normalsize\bfseries}{\thesubsection.}{1 ex}{}
\titleformat{\subsubsection}[runin]
	{\normalfont\normalsize\bfseries\filcenter}{\thesubsection.}{1 ex}{}

\usepackage[charter]{mathdesign}
\usepackage[mathcal]{eucal}

\usepackage[margin=1.2 in]{geometry}

\usepackage[square,numbers,sort&compress]{natbib}

\usepackage{thmtools,thm-restate}
\declaretheorem[within=section]{theorem}
\declaretheorem[within=section]{lemma}
\declaretheorem[sibling=lemma,name=Proposition]{prop}

\newcommand{\iprod}[2]{\left\langle {#1},{#2}\right\rangle}

\newcommand{\eop}{\hfill$\qed$}

\newcommand\hatP{\hat{P}}

\newcommand\hatp{\hat{p}}

\newcommand\hatx{\hat{x}}

\newcommand\cV{\mathcal{V}}

\renewcommand\AA{\mathbb{A}}

\newcommand\EE{\mathbb{E}}

\newcommand\NN{\mathbb{N}}
\newcommand\QQ{\mathbb{Q}}
\newcommand\RR{\mathbb{R}}

\newcommand{\Gale}[1]{{#1}^{\vee}} 

\newcommand{\trans}[1]{{#1}^\top}

\DeclareMathOperator{\conv}{conv}

\DeclareMathOperator{\rank}{rank}

\newcommand{\defn}[1]{\emph{\color{blue} #1}} 

\newcommand{\QQQ}{{\cal Q}}
\newcommand{\CCC}{{\cal C}}

\newcommand*{\doi}[1]{doi: \href{https://dx.doi.org/#1}{\urlstyle{rm}\nolinkurl{#1}}}
\newcommand*{\arxiv}[1]{arXiv:  \href{https://arxiv.org/abs/#1}{\urlstyle{rm}\nolinkurl{#1}}}
\usepackage{verbatim}
\begin{document}
\title{Global rigidity of complete bipartite graphs}
\author{Robert Connelly\thanks{Department of Mathematics, Cornell University, Ithaca, USA. Partially supported by NSF Grant: DMS-1564493.} \and Steven J. Gortler\thanks{School of Engineering and Applied Sciences, Harvard University, Cambridge, USA. Partially supported by NSF Grant: DMS-1564473 and the Aalto Science Institute (AScI) Thematic Program 
``Challenges in large geometric structures and big data''.} \and Louis Theran\thanks{School of Mathematics and Statistics, University of St Andrews, St Andrews, Scotland. Partially supported by Finnish Academy (AKA) project COALESCE.}}

\date{}

\maketitle 
\section{Introduction}\label{sec:intro}
In this note, we prove the following.
\begin{theorem}\label{thm: bipartite GGR}
Let $d\in \NN$ and $m,n \ge d + 1$, with $m + n \ge \binom{d+2}{2} + 1$.  Then the complete bipartite graph 
$K_{m,n}$ is generically globally rigid in dimension $d$.
\end{theorem}
This statement has appeared in \cite[Theorem 63.2.2]{HANDBOOK}, but a 
proof hasn't yet been circulated.

\section{Setup and background}\label{sec:setup}
We start by introducing the necessary concepts and definitions.

\subsection{Rigidity}

\paragraph{Frameworks}
A framework $(G,p)$ a graph $G$ with $n$ vertices and a 
\defn{configuration} $p : V\to \EE^d$, mapping the vertex set $V$ of 
$G$ to a $d$-dimensional point set in Euclidean space.

By picking an origin arbitrarily, we identify points  $x\in \EE^d$ 
with  affine coordinates of the form 
$\hatx := (\cdots, 1)\in \RR^{d+1}$.  Thus, we may identify a configuration
with a vector in $\left(\RR^{d}\right)^n$ or its affine counterpart 
$\hatp\in \left(\RR^{d+1}\right)^n$. we can also write this as an
$n\times (d+1)$ \defn{configuration matrix} $\hatP$.

Fix a dimension $d$ and a graph $G$.  Two frameworks $(G,p)$ and 
$(G,q)$ are \defn{equivalent} if
\[
    \|p(j) - p(i)\| = \|q(j) - q(i)\|
    \qquad \text{(all edges $\{i,j\}$ of $G$)}
\]
They are \defn{congruent} if there is a Euclidean motion 
$T$ of $\EE^d$ so that
\[
    q(i) = T(p(i))\qquad \text{(all verts. $i$ of $G$)}
\]
A framework $(G,p)$ is \defn{rigid} if there is a neighborhood $U\ni p$
so that if $q\in U$ and $(G,q)$ is equivalent to $(G,p)$, then $q$ is 
congruent to $p$.  A framework $(G,p)$ is \defn{globally rigid} if \emph{any}
$(G,q)$ equivalent to $(G,p)$ is congruent to it.  

Rigidity \cite{AR78} and global rigidity \cite{GHT10} are 
\defn{generic properties}.  A configuration is \defn{generic} if its 
coordinates are algebraically independent over $\QQ$.
\begin{theorem}[\cite{AR78,GHT10}]
Let $d$ be a dimension and $G$ a graph.  Then either every generic 
framework $(G,p)$ in dimension $d$ 
is (globally) rigid or no generic framework is 
(globally) rigid.
\end{theorem}
If every generic framework $(G,p)$ in dimension $d$ is globally 
rigid, we say that $G$ is \defn{generically globally rigid (GGR)}
in dimension $d$.

\paragraph{Infinitesimal rigidity}
The \defn{rigidity matrix} $R(p)$ of a framework $(G,p)$ is the 
matrix of the linear system
\[
    \iprod{p(j) - p(i)}{p'(j) - p'(i)} = 0
    \qquad \text{(all edges $\{i,j\}$ of $G$)}
\]
where the vector configuration $p'$ is variable.  The kernel
of $R(p)$ comprises the \defn{infinitesimal flexes} of $(G,p)$.
When $G$ is a graph with $n\ge d$ vertices, a $d$-dimensional 
framework $(G,p)$ is called \defn{infinitesimally rigid}
when $R(p)$ has rank $dn - \binom{d+1}{2}$.  Infinitesimal rigidity
implies rigidity \cite{AR78}.

\paragraph{Generic global rigidity}
The main tool we will use in this paper to prove that a graphs are GGR is
the following:
\begin{theorem}[\cite{GHT10}]\label{thm:infggr}
Let $G$ be a graph and $d$ a dimension. Suppose that there is a framework 
$(G,p)$ that is infinitesimally rigid and globally rigid in dimension $d$.  
Then $G$ is GGR in dimension $d$.
\end{theorem}

To construct  frameworks that are
globally rigid, we   use the
stronger property of universal rigidity.
A framework $(G,p)$ is \defn{universally rigid} 
if any equivalent framework \emph{in any dimension} 
is congruent. 
One important way to certify that 
a constructed framework is universally rigid is via
the still stronger property of super stability.
To define  this we need a bit more terminology.

\paragraph{Equilibrium stresses}
For a graph $G$, define the space $S(G)$ of \defn{graph supported matrices}
to be the symmetric $n\times n$ matrices that have zeros in the off-diagonal 
entries indexed by non-edges $\{i,j\}$.  A matrix $\Omega\in S(G)$ is a 
\defn{stress matrix} if it has the vector of all ones in its kernel. 
A stress matrix $\Omega$ is an \defn{equilibrium stress matrix} of a 
framework $(G,p)$ if
$\Omega\hatP = 0$.  A computation shows that for each vertex $i$
\[
	\sum_{j\neq i} \Omega_{ij} [p(j) - p(i)] = 0
\]
if and only if a stress matrix $\Omega$ is an equilibrium stress matrix for $(G,p)$.  
Thus, equilibrium stress matrices
are obtained by re-arranging \defn{equilibrium stresses} of $(G,p)$, which 
are vectors $\omega$ in the cokernel of $R(p)$.  

Suppose that $n \ge d$, let $(G,p)$ be a $d$-dimensional framework 
and denote by $m$ the number of edges, $r$ the rank of the rigidity matrix, 
$s$ the dimension of the space of equilibrium stresses and $f$ the 
dimension of infinitesimal flexes.  Linear algebra duality gives
us the Maxwell index theorem:
\begin{equation}\label{eq: maxwell idx}
    m - r = s - f + \binom{d+1}{2}
\end{equation}
We then see that $(G,p)$ is infinitesimally rigid if and 
only if $s = m - dn + \binom{d+1}{2}$.

\paragraph{Super stability}
The \defn{edge directions} of a framework $(G,p)$ is the configuration $e$  
of $|E|$ points at infinity $e(i,j) := p(j) - p(i)$. 
 A framework \defn{has its edge directions on a conic at infinity}
if there is a quadric surface $\QQQ$ at infinity 
containing all of $e$.  A framework
with $d$-dimensional affine span 
is \defn{super stable} if it has a positive semidefinite (PSD) 
equilibrium stress matrix $\Omega$ of rank $n - d - 1$ 
and its edges directions are not on a conic at infinity.  

The main connection between these concepts is due to Connelly.
\begin{theorem}[\cite{C82}]\label{thm:sstour}
If $(G,p)$ is super stable, then it is universally rigid.
\end{theorem}

\subsection{Bipartite graphs and partitioned point sets}

We are interested in graphs $G$ that are simple and bipartite, 
with vertex partition $U$, $V$ and edge set $E$.  
We denote by $u$ and $v$ the size of $U$ 
and $V$, respectively, and the total number of vertices by $n := u + v$.
The number of edges is $m := |E|$.

For notational convenience, we denote a configuration of 
the vertices of a bipartite graph
by a pair of mappings $p : U \to \EE^d$ and $q : V\to \EE^d$,
and a framework on a bipartite graph by $(G,p,q)$.  All the 
other definitions discussed in the previous section for point 
configurations extend naturally to partitioned point configurations
$(p,q)$.

\subsection{General position and quadric separability}
We say that a point set $p$ of at least $d+1$ points 
in dimension $d$ is in 
\defn{(affine) general position} if 
any $d+1$ of the points are affinely independent.
This is equivalent to any $d+1$ of the vectors in $\hat p$
being linearly independent.

Fix a dimension $d\in \NN$ and set $D = \binom{d+2}{2} - 1$.
Let $\cV : \RR^{d+1}\to \RR^{D+1}$ denote the degree $2$ 
homogeneous \defn{Veronese map}, which is defined by $x \to x\trans{x}$.
This is well-defined, since the image is a subset of
symmetric $(d+1)\times (d+1)$ matrices, which are
naturally identified with $\RR^{D + 1}$ by a suitable choice
of coordinates.  
We give this $\RR^{D+1}$ the trace inner product 
\[
    \iprod{X}{Y} = \operatorname{Tr}(XY)
\]

We define the action of $\cV$ on $x\in \EE^d$ by 
$\cV(x) = \hatx\trans{\hatx}$.  This extends 
to an action $\cV(p)$ on point configurations including 
partitioned point configurations.
For $x\in \EE^d$, $\cV(x)$ is a matrix with a 
$1$ in the bottom right corner.  Thus, 
we can view $\cV$ as mapping $\EE^d$ to 
a $D$-dimensional affine space, which we denote 
by $\AA^D$.

The inner product described above identifies 
$\left(\RR^{D+1}\right)^*$ with quadratic polynomials 
on $\EE^d$, since, if $Q$ is a $(d+1)\times (d+1)$ symmetric 
matrix, 
\[
    \trans{\hatx} Q \hatx = \operatorname{Tr}\left((\hatx\trans{\hatx})Q\right)
    = \iprod{\cV(x)}{Q}
\]
We note that the identification implies the following, which we need later:
\begin{lemma}\label{lem: veronese nondegen}
Let $d\in \NN$ be a dimension.  Then $\cV(\EE^d)$ affinely spans
$\AA^D$.
\end{lemma}
\begin{proof}
If $\cV(\EE^d)$ has defective affine span, then there is a non-zero
quadratic polynomial vanishing on all of $\EE^d$, which is 
impossible.
\end{proof}

A partitioned point configuration 
$(p,q)$ is \defn{strictly quadratically separable} if $\cV(p)$ 
and $\cV(q)$ are strictly  separable 
by an (affine) hyperplane in $\AA^{D}$.  This is equivalent to there being a 
quadric surface $\QQQ$ in $\EE^d$ strictly 
separating the points of $p$ from those of $q$.

We recall that  $\cV(p)$ and $\cV(q)$ are not strictly  
separable by a hyperplane in $\AA^D$, if and only if their 
convex hulls $\conv(\cV(p))$ and $\conv(\cV(q))$ 
have non-empty intersection.

\section{Super-stable realizations of $K_{d+1,d+1}$}\label{sec:kd+1d+1}

In this section we construct super stable realizations 
of $K_{d+1,d+1}$ in dimension $d$ with some additional
properties. 
\begin{prop}\label{prop:Kd+1d+1}
For each $d$, there are $(p,q)$ such that: (a) $p$ and $q$ both have 
$d$-dimensional affine span in $\EE^d$; (b) $\cV(p,q)$ has 
$2d+1$-dimensional linear span in $\RR^{D+1}$; (c) $(K_{d+1,d+1},p,q)$ 
is super stable.
\end{prop}
The key rigidity theoretic tool we need 
is a  result of Connelly and Gortler.
\begin{theorem}[\cite{CG15}]\label{thm:bipartite}
Let $u\ge v\ge 1$.  If $\conv(\cV(p))$ and $\conv(\cV(q))$
intersect in their relative interiors,  then $(K_{u,v},p,q)$ 
is super stable.  If $\conv(\cV(p))$ and $\conv(\cV(q))$ are 
disjoint, then $(K_{u,v},p,q)$ is not universally rigid.
\end{theorem}
The rest of this section builds the proof of Proposition 
\ref{prop:Kd+1d+1} in stages.

\paragraph{Not quadratically separable}
Let $\CCC$ be a curve in $\EE^d$. We assume there is a parameterization
$f_\CCC(t)$.  Suppose that $(p,q)$ is on $\CCC$, and define $(p,q)$
to be \defn{alternating on $\CCC$} if there are $s_i$ and $t_j$ such that
$s_1 <  t_1 <  s_2 <  \cdots$, with $p(i) = f_\CCC(s_i)$ and $q(j) = f_\CCC(t_j)$. 

\begin{lemma}\label{lem:unsep}
If $u\ge v\ge d+1$, and $(p,q)$ is alternating on a degree $d$ curve $\CCC$, then $(p,q)$ is not strictly quadratically separable.
\end{lemma}
\begin{proof}
The alternating property implies that any separating quadric 
$\QQQ$ would have to intersect $\CCC$ transversely in at least $2d+1$ points.  Since $\CCC$ is degree $d$, 
Bezout's Theorem implies that, 
for any quadric $\QQQ$, either $|\CCC\cap \QQQ| \le 2d$ or 
$\QQQ$ contains a component of $\CCC$. Either case 
contradicts strict separation.
\end{proof}

\paragraph{General position}
We will place
$(p,q)$ so that it alternates along the 
\defn{rational normal curve} $\CCC_d$, which is the
projective counterpart to the moment curve.  For $s, t $ real, 
\[
  [s : t] \mapsto [t^d : t^{d-1}s : \cdots : ts^{d-1} : s^d].
\]
The part of $\CCC_d$ in $\EE^d$ 
is the moment curve, obtained by setting $s=1$.

The rational normal curve $\CCC_d$ is characterized by the property that 
any $d+1$ points on it are affinely independent.  
We need a similar statement for the re-embedded curve $\cV(\CCC_d)$, which 
has the parameterization 
\[
     [s : t] \mapsto \cV([t^d : t^{d-1}s : \cdots : ts^{d-1} : s^d])
\]
It is immediate that the degree of $\cV(\CCC_d)$ is $2d$.  Looking more 
closely, we see that, in fact, $\cV(\CCC_d)$ is a rational normal curve 
in its affine span.\footnote{We thank Jessica Sidman for suggesting this approach.}
The $ij$th entry of $\cV([t^d : t^{d-1}s : \cdots : ts^{d-1} : s^d])$
is 
\[
    t^{d - i + 1}s^{i-1}t^{d - j + 1}s^{j-1} = t^{2d - i - j + 2}s^{i+j - 2}
\]
Since every entry is determined by $i + j$, there are $2d+1$ 
distinct entries, and any set of these parameterizes a rational normal 
curve of degree $2d$.
From this, we see that any $2d+1$ points on $\cV(\CCC_d)\cap \AA^D$
are affinely independent.

\begin{lemma}\label{lem:gp}
Suppose that 
that $(p,q)$ is alternating along the 
$d$-dimensional rational normal curve and 
each of $p$ and $q$ have $d+1$ points in $\EE^d$.
Then $p$ and $q$ are in affine 
general position in $\EE^d$ and, under the 
Veronese map, any $2d + 1$ 
points of $(p,q)$ have $2d$-dimensional affine 
span in $\AA^{D}$.
\end{lemma}

\paragraph{Proof of Proposition \ref{prop:Kd+1d+1}}
Parts (a) and (b) are Lemma \ref{lem:gp}.

For part (c), we know from Lemma \ref{lem:unsep} and the 
alternating pattern, that $P = \conv(\cV(p))$ and 
$Q = \conv(\cV(q))$ have non-empty intersection.  To conclude 
super-stability from Theorem \ref{thm:bipartite}, 
we need that $P$ and $Q$ intersect in their relative 
interiors.  

Suppose the contrary for a contradiction.
Let $P'$ and $Q'$ be maximal faces of 
$P'$ and $Q'$ that meet in their relative interiors. Necessarily, 
$P'$ and $Q'$ are proper, so they span at most $2d$ points in total.
Writing any $x\in P'\cap Q'$ as a convex combination of 
vertices of $P'$ and $Q'$, respectively, shows that these 
points are affinely dependent in $\AA^D$.  This 
contradicts Lemma \ref{lem:gp}. Hence we conclude that $P$ and 
$Q$ intersect in their relative interiors, as desired.
\eop

\section{The proof}\label{sec:mainproof}
For the rest of this section let $d$, $m$ and $n$ 
be as in the statement.  
We refer to the subgraph induced by the first $d+1$ vertices 
in each part as the 
\defn{core} of $K_{m,n}$, and denote by $U_1$ and $V_1$ 
the vertices of the core.  
For notational convenience, we also 
denote by $p(U')$ and $q(V')$ the sub-configurations of $p$ and
$q$ indexed by $U'\subseteq U$ and $V'\subseteq V$.

The proof strategy is to start from Proposition \ref{prop:Kd+1d+1}  
to produce a configuration $(p,q)$ such that 
$(G,p,q)$ is infinitesimally rigid and 
globally rigid.  The desired statement then follows from Theorem \ref{thm:infggr}.  

\subsection{The Bolker-Roth stress decomposition}
We will use 
results of Bolker and Roth \cite{BR80}
on the stresses of 
complete bipartite graphs.
Denote by $\Gale{p}$ the \defn{Gale dual} (see, e.g., \cite[Section 6.3]{Z95}) 
of the vector configuration $\hatp$.  The \defn{rank} of $\Gale{p}$ is equal 
to the dimension of the cokernel of the configuration matrix $\hatP$.  We similarly define the rank of $\Gale{(\cV(p,q))}$ in terms of the images in 
$\RR^{D+1}$ of $p$ and $q$ under $\cV$.
\begin{theorem}[\cite{BR80}]\label{thm:bolker-roth}
Suppose that $p$ and $q$ both have full affine span in $\EE^d$.  The 
the dimension of the space of equilibrium-stresses of $(K_{u,v},p,q)$ is 
given by
\(
	\rank(\Gale{p})\rank(\Gale{q}) + \rank(\Gale{(\cV(p,q))})
\).
\end{theorem}

\subsection{Trilateration and global rigidity}
A graph $H$ is a \defn{trilateration in dimension $d$} 
of a graph $G$, if $H$ is obtained from $G$ by adding a 
new vertex to $G$ and connecting it to at least $d+1$
neighbors.

Global and universal rigidity are very well-behaved with respect to trilateration: it is 
preserved no matter where the new vertex is placed, so 
long as the neighbors affinely span $\EE^d$.  
This seems to be a folklore result.
\begin{lemma}\label{lem:trilatgr}
Let $(G,p)$ be a globallly (resp universally) rigid 
framework in dimension $d$.  Let $H$ be a graph obtained 
by trilaterating $G$ and $U$ the set of at least $d+1$
neighbors.  If $p(U)$ affinely spans $\RR^d$, then 
for any placement $p(v_0)$ of the new vertex $v_0$, 
the resulting framework on $H$ is globally (resp universally) 
rigid.
\end{lemma}
\begin{proof}[Sketch]
This follows from the fact that $K_{d+2}$ is universally rigid in 
dimension $d$, and gluing two globally (or universally) rigid 
frameworks along $d+1$ affinely independent 
vertices in dimension $d$
preserves global (or universal) rigidity.
\end{proof}
We note that an immediate consequence is that 
trilaterating a generically globally 
rigid graph yields another generically globally rigid graph.

\subsection{Proof of Theorem \ref{thm: bipartite GGR}}
By Lemma \ref{lem:trilatgr}, it is sufficient to prove that $K_{m,n}$ is
generically globally rigid when $m,n\ge d+1$ and 
$m + n = \binom{d+2}{2} + 1 = D+2$.
From now on, we make this assumption.

\paragraph{The construction}
Now we describe the construction of $(p,q)$.  Realize the core of $K_{m,n}$
using Proposition \ref{prop:Kd+1d+1}.  Then, at each trilateration 
step, place the new vertex generically.

\paragraph{Infinitesimally rigid}
Using $m + n = \binom{d+2}{2} + 1$, we have, by direct computation 
\[
    mn - (m - d -1)(n - d -1) - 1 = d(m+n) - \binom{d+1}{2} 
\]
The Maxwell index theorem (Equation \eqref{eq: maxwell idx})
then tells us that if the dimension of the 
space of equilibrium stresses of $(K_{m,n},p,q)$ has dimension
\[
    (m - d -1)(n - d -1) + 1
\]
then $(K_{m,n},p,q)$ is infinitesimally rigid.

The desired equilibrium stress space dimension follows from 
Theorem \ref{thm:bolker-roth} and two observations:
\begin{itemize}
    \item Since the core is realized in general position and the rest of the points 
        are placed generically, $(p,q)$ is in general position.  Hence 
        $\rank(\Gale{p})\rank(\Gale{q}) =  (m - d -1)(n - d -1)$.
    \item The core is realized (non-generically) so that 
        $\rank(\Gale{(\cV(p(U_1),q(V_1)))}) = 1$.  This is 
        because the $2d + 2$ points of 
        $\cV(p(U_1),q(V_1))$
        have $2d$-dimensional affine span in $\AA^D$ by 
        Proposition \ref{prop:Kd+1d+1}.
        (And so have one non-trivial affine dependency in $\AA^D$.)
        
        We add $D + 2 - 2(d+1)$ additional generic points in $\EE^d$ during the 
        trilateration phase.  The images of these under $\cV$ are generic in 
        $\cV(\EE^d)$, which has $D$-dimensional affine span in $\AA^D$
        (Lemma \ref{lem: veronese nondegen}),
        so no new affine dependencies appear during the trilateration phase.
        Hence, the $D+2$ points of $\cV(p,q))$ affinely span $\AA^D$, 
        which implies that  $\rank(\Gale{(\cV(p,q))}) = 1$.
\end{itemize}

\paragraph{Globally rigid} 
Proposition~\ref{prop:Kd+1d+1}
implies that the core is super stable and thus universally rigid.
Lemma \ref{lem:trilatgr} then implies that $(G,p,q)$ is universally 
rigid and thus globally rigid.
\eop \eop

\section{Concluding remarks}
The  hypotheses of Theorem \ref{thm: bipartite GGR} are also necessary.  
We need $m, n\ge d+1$ 
to avoid contradicting Hendrickson's necessary conditions \cite{H92} for a 
graph to be GGR.
The classification of equilibrium stresses in complete bipartite 
graphs by Bolker and Roth \cite{BR80} implies that, if $m + n < \binom{d+2}{2} + 1$ 
any equilibrium stress matrix of a generic framework $(K_{m,n},p,q)$ has all zeros on its diagonal and
is thus necessarily of deficient rank.

\section*{Acknowledgments}
We thank Walter Whiteley for asking if there is an explicit
construction of super stable configurations of complete 
bipartite GGR graphs and Jessica Sidman for fielding 
questions about the rational normal curve.


\begin{thebibliography}{8}
\providecommand{\natexlab}[1]{#1}
\providecommand{\url}[1]{\texttt{#1}}
\expandafter\ifx\csname urlstyle\endcsname\relax
  \providecommand{\doi}[1]{doi: #1}\else
  \providecommand{\doi}{doi: \begingroup \urlstyle{rm}\Url}\fi

\bibitem[Asimow and Roth(1978)]{AR78}
L.~Asimow and B.~Roth.
\newblock The rigidity of graphs.
\newblock \emph{Trans. Amer. Math. Soc.}, 245:\penalty0 279--289, 1978.
\newblock \doi{10.2307/1998867}.

\bibitem[Bolker and Roth(1980)]{BR80}
E.~D. Bolker and B.~Roth.
\newblock When is a bipartite graph a rigid framework?
\newblock \emph{Pacific J. Math.}, 90\penalty0 (1):\penalty0 27--44, 1980.
\newblock URL \url{http://projecteuclid.org/euclid.pjm/1102779115}.

\bibitem[Connelly(1982)]{C82}
R.~Connelly.
\newblock Rigidity and energy.
\newblock \emph{Invent. Math.}, 66\penalty0 (1):\penalty0 11--33, 1982.
\newblock \doi{10.1007/BF01404753}.

\bibitem[Connelly and Gortler(2017)]{CG15}
R.~Connelly and S.~J. Gortler.
\newblock Universal rigidity of complete bipartite graphs.
\newblock \emph{Discrete Comput. Geom.}, 57\penalty0 (2):\penalty0 281--304,
  2017.
\newblock \doi{10.1007/s00454-016-9836-9}.

\bibitem[Goodman et~al.(2018)Goodman, O'Rourke, and T\'{o}th]{HANDBOOK}
J.~E. Goodman, J.~O'Rourke, and C.~D. T\'{o}th, editors.
\newblock \emph{Handbook of discrete and computational geometry}.
\newblock Discrete Mathematics and its Applications (Boca Raton). CRC Press,
  Boca Raton, FL, 2018.

\bibitem[Gortler et~al.(2010)Gortler, Healy, and Thurston]{GHT10}
S.~J. Gortler, A.~D. Healy, and D.~P. Thurston.
\newblock Characterizing generic global rigidity.
\newblock \emph{Amer. J. Math.}, 132\penalty0 (4):\penalty0 897--939, 2010.
\newblock \doi{10.1353/ajm.0.0132}.

\bibitem[Hendrickson(1992)]{H92}
B.~Hendrickson.
\newblock Conditions for unique graph realizations.
\newblock \emph{SIAM J. Comput.}, 21\penalty0 (1):\penalty0 65--84, 1992.
\newblock \doi{10.1137/0221008}.

\bibitem[Ziegler(1995)]{Z95}
G.~M. Ziegler.
\newblock \emph{Lectures on polytopes}, volume 152 of \emph{Graduate Texts in
  Mathematics}.
\newblock Springer-Verlag, New York, 1995.
\newblock \doi{10.1007/978-1-4613-8431-1}.

\end{thebibliography}

\end{document}